\documentclass{amsart}
\usepackage{fullpage}
\usepackage{yhmath}  
\usepackage{amsmath,amsthm,amssymb}
\usepackage{marvosym}
\usepackage{graphicx}
\usepackage{color}
\usepackage{enumerate}
\usepackage{cite}
\usepackage{tabu}

\usepackage[active]{srcltx}

\theoremstyle{plain}
\newtheorem{theorem}{Theorem} 

\newtheorem{corollary}[theorem]{Corollary}
\newtheorem{conjecture}[theorem]{Conjecture}

\newtheorem{lemma}[theorem]{Lemma}

\newtheorem{question}[theorem]{Question}

\newtheorem*{claim*}{Claim}

\theoremstyle{definition}

\newtheorem{remark}[theorem]{Remark}

\newcommand{\comment}[1]{}

\newcommand{\bdry}{\ensuremath{\partial}}

\begin{document}
\baselineskip 14pt

\title{A note on the concordance of fibered knots}

\author{Kenneth L. Baker}
\address{Department of Mathematics, University of Miami, 
Coral Gables, FL 33146, USA}
\email{k.baker@math.miami.edu}


\thanks{KB is partially supported by a grant from the Simons Foundation (\#209184 to Kenneth L.\ Baker).}

\begin{abstract}
Either fibered knots supporting the tight contact structure are unique in their smooth concordance class or there exists a fibered counterexample to the Slice-Ribbon Conjecture.
\end{abstract}
\maketitle

Rudolph questioned whether algebraic knots are linearly independent in the concordance group \cite{rudolphquestion}.  We conjecture something stronger.
\begin{conjecture}\label{conj:concordance}
Let $K_0$ and $K_1$ be fibered knots in $S^3$ supporting the tight contact structure.  If $K_0$ and $K_1$ are concordant, then $K_0 = K_1$.
\end{conjecture}

While we expect an affirmative answer to Rudolph's question and our conjecture in the locally-flat category, in this note we content ourselves with the smooth category.

Knots $K_0$ and $K_1$ in $S^3$ are (smoothly) {\em concordant} if there exists an embedded  smooth annulus in $S^3 \times[0,1]$ connecting the knots $K_i \subset S^3 \times\{i\}$.   This is equivalent to $K_0 \# -K_1$ being (smoothly) slice, i.e.\ the boundary of a properly embedded smooth disk in $B^4$. Such a disk is a ribbon disk if it has no maxima with respect to the standard radial function on $B^4$. The still unsettled Slice-Ribbon Conjecture proposes that every slice knot actually bounds a ribbon disk \cite{fox}.

Any knot $K$ in $S^3$ satisfies $\tau(K) \leq g_4(K) \leq g(K)$.  Here $\tau(K)$ is the Ozsv\'{a}th-Szab\'{o}'s concordance invariant \cite{ozsztau}, $g_4(K)$ is the smooth slice genus, and $g(K)$ is the Seifert genus.  For strongly quasipositive knots these three quantities are all equal according to Livingston \cite{livingston}.  Hedden showed that for a fibered knot $K$ in $S^3$, being strongly quasipositive, supporting the tight contact structure, and satisfying $\tau(K)=g_4(K)=g(K)$ are all equivalent \cite{heddennotionsofpositivity}.   It follows that if $K_0$ and $K_1$ are concordant fibered knots in $S^3$ supporting the tight contact structure, then $g(K_0)=g(K_1)$.

\medskip
Conjecture~\ref{conj:concordance} is supported by Theorem~\ref{thm:ribbon} below which is based on Miyazaki's work on homotopy ribbon concordance of fibered knots \cite{miyazaki-nonsimpleribbon}.

For knots $K_i$ in homology $3$--spheres $M_i$, $i=0,1$, $(M_1, K_1)$ is {\em homotopy ribbon concordant} to $(M_0,K_0)$ (denoted $(M_1, K_1) \geq (M_0,K_0)$ or just $K_1 \geq K_0$) if some $4$--manifold $X$ is a homology $S^3 \times [0,1]$ containing an embedded smooth annulus $A$ such that $(\bdry X, A \cap \bdry X) = (M_1, K_1) \sqcup -(M_0,K_0)$ with surjective $\pi_1 (M_1 - K_1) \to \pi_1(X-A)$ and injective  $\pi_1 (M_0 - K_0) \to \pi_1(X-A)$.  Homotopy ribbon concordance is a generalization of ribbon concordance \cite[Lemma 3.1]{gordon-ribbonconcordaneofknots}.

Also note that the concordance invariance of $\tau$ holds more generally  whenever the concordance annulus between $(S^3,K_0)$ and $(S^3, K_1)$ is in a homology $S^3 \times [0,1]$ with boundary $(S^3, K_1) \sqcup -(S^3,K_0)$ (via Theorem 1.1 \cite{ozsztau}).  Thus, for knots in $S^3$, $K_1 \geq K_0$ implies $\tau(K_1) = \tau(K_0)$.

For short, let us say a knot is {\em tight fibered} if it is a fibered knot supporting a tight contact structure.
\begin{lemma}\label{lem:ribbonconcmin}
Let $K$ be a tight fibered knot in $S^3$.  Then $K$ is minimal with respect to homotopy ribbon concordance among fibered knots in $S^3$.
\end{lemma}
\begin{proof}
If a fibered knot $K$ in $S^3$ is homotopically ribbon concordant to a fibered knot $J$ in $S^3$, then Gordon's Lemma 3.4 \cite{gordon-ribbonconcordaneofknots} (which holds for homotopy ribbon concordance) implies either $J=K$ or $g(K)>g(J)$ because the Alexander polynomial detects the genus of a fibered knot.   However since $K$ is a tight fibered knot with a homotopy ribbon concordance to $J$, we have $g(K) = \tau(K) = \tau(J) \leq g(J)$.  Hence $J=K$.
\end{proof}

Using Lemma~\ref{lem:ribbonconcmin}, 
Theorem 5.5 of Miyazaki \cite{miyazaki-nonsimpleribbon} specializes to yield:
\begin{theorem}\label{thm:ribbon}
Let $K_0$ and $K_1$ be tight fibered knots in $S^3$.   If $K_0 \# -K_1$ is ribbon, then $K_0 = K_1$.  \qed
\end{theorem}

\begin{proof}
The purpose of Condition (1) of Theorem 5.5 of  \cite{miyazaki-nonsimpleribbon} is to conclude the hypotheses of \cite[Lemma~1.2]{miyazaki-nonsimpleribbon} do not hold.   That is, if a fibered knot $K$ is minimal with respect to $\geq$ among fibered knots in homology spheres,
then the monodromy of  $K$ does not extend across a non-trivial compression body with connected positive and negative boundaries.

Using Perelman's resolution of the Geometrization Conjecture \cite{perelman1,perelman2,perelman3} which implies that $3$--manifolds groups are residually finite, we can specialize Condition (1) so that we may invoke Lemma~\ref{lem:ribbonconcmin}.  
If the hypotheses of \cite[Lemma~1.2]{miyazaki-nonsimpleribbon} hold for a fibered knot $K$ in $S^3$, then there is a fibered knot $K'$ in a homology sphere $M'$ such that $(S^3, K) \geq (M',K')$.  Furthermore, $M'=S^3$ by \cite[Lemma~1.3]{miyazaki-nonsimpleribbon} and Geometrization; see the remark after \cite[Lemma~1.3]{miyazaki-nonsimpleribbon}. 
  Hence we may replace Condition (1) of Theorem 5.5 of  \cite{miyazaki-nonsimpleribbon} with the condition that a fibered knot $K$ is minimal with respect to $\geq$ among fibered knots in $S^3$.

Now since $K_0 \# -K_1$ being ribbon implies $K_0 \# -K_1 \geq 0$, Lemma~\ref{lem:ribbonconcmin} and the use of our specialized Condition (1) in Theorem 5.5 of  \cite{miyazaki-nonsimpleribbon} yields the result.  
\end{proof}

\begin{corollary}
Either Conjecture~\ref{conj:concordance} is true or the Slice-Ribbon Conjecture is false. \qed
\end{corollary}
Tetsuya Abe and Keiji Tagami observed that since connected sums of tight fibered knots are again tight fibered knots, we can similarly approach Rudolph's question about algebraic knots for the broader class of prime tight fibered knots. 
\begin{corollary}[Lemma~3.1 \cite{abe}]
Either the prime tight fibered knots are linearly independent in the concordance group or the Slice-Ribbon Conjecture is false. \qed
\end{corollary}

\begin{remark}
Indeed, as Katura Miyazaki has pointed out to me, Theorem~\ref{thm:ribbon} and its consequences hold for any class of fibered knots in $S^3$ that satisfy the specialized Condition (1), not just the class of tight fibered knots \cite{miyazaki-personalcomm}.
\end{remark}

Viewing strongly quasipositive knots as a generalization of tight fibered knots, one may wonder if Conjecture~\ref{conj:concordance} actually holds for these knots too. Matt Hedden, however, has shown me a simple construction that demonstrates Conjecture~\ref{conj:concordance} is not true for strongly quasipositive knots in general \cite{hedden-personalcomm}. Take any two non-isotopic ribbon concordant knots whose maximum Thurston-Bennequin numbers are at least some integer $N$.   Then for any integer $n\leq N$, the $n$--twisted positive-clasped Whitehead doubles of the two knots will be strongly quasipositive and ribbon concordant but not isotopic.   Rudolph determines the strong quaispositivity of these Whitehead doubles \cite[102.4]{rudolph2}.  By running $(\mbox{solid torus}, \mbox{pattern knot}) \times [0,1]$ along a regular neighborhood of a concordance annulus, satellite operations preserve the relations of concordance and ribbon concordance, see for example  \cite[Lemmas 3.1 and 3.2]{kawauchi}.  Finally, by considering JSJ-decompositions of knot exteriors \cite{JSXdecomp,XXJdecomp}, one observes that a satellite operation on non-isotopic knots {\em usually} produces non-isotopic knots.  Since we are taking twisted Whitehead doubles, our resulting satellite knots will indeed be non-isotopic \cite{kouno-motegi}.  

Hedden suggests that positive knots are a restricted subset of strongly quasipositive knots containing many non-fibered knots for which Theorem~\ref{thm:ribbon} and Conjecture~\ref{conj:concordance} potentially generalize.  Stoimenow has presented a related conjecture \cite{stoimenow}.
\begin{question}
For positive knots $K_0$ and $K_1$, if $K_0 \# -K_1$ is ribbon then must $K_0 = K_1$?   What if $K_0 \# -K_1$ is slice?
\end{question}

Let us end with a couple of brief remarks on algebraic concordance. Hedden-Kirk-Livingston have shown that there are many pairs of  tight fibered knots (sums of algebraic knots) that are algebraically concordant and yet not topologically locally-flat concordant \cite{HKL}.
Meier also points out that there are many pairs of positive non-fibered knots that are algebraically concordant and yet not topologically locally-flat concordant \cite{meier-personalcomm}; as reported by KnotInfo and its Concordance Calculator, the positive knots $3_1 \# 9_2$ and $9_{23}$ give  an example \cite{knotinfo, lowcrossingconcordance}. 

\bigskip

We gratefully acknowledge Tetsuya Abe, Cameron Gordon, Matt Hedden, Jeff Meier, Katura Miyazaki, Kimihiko Motegi, and Nikolai Saveliev for useful conversations and their interest.   We are particularly indebted to Matt Hedden for sharing his knowledge about the failure of Conjecture~\ref{conj:concordance} for strongly quasipositive knots in general.  Furthermore, we deeply thank Katura Miyazaki for pointing out and correcting an error in our original argument for Theorem~\ref{thm:ribbon}.   

This work was partially supported by Simons Foundation Collaboration Grant \#209184 to Kenneth L.\ Baker.

\bibliographystyle{amsalpha}
\bibliography{fiberedconcordance}
\end{document}